\documentclass[11pt,a4paper]{article}
\usepackage[OT1]{fontenc}
\usepackage[english]{babel}
\usepackage{amsfonts}
\usepackage{amsthm}
\usepackage{amsmath}

\def\h{ {\cal H} }

\def\a{ {\cal A} }

\def\b{ {\cal B} }
\def\u{ {\cal U} }
\def\v{ {\cal V} }

\def\ii{ {\cal I} }

\def\s{ {\cal S} }

\def\p{ {\cal P} }

\def\o{ {\cal O} }
\def\xx{ {\bf x} }

\def\vv{ {\bf v} }

\def\XX{ {\bf X} }

\def\YY{ {\bf Y} }

\def\AA{ {\bf A} }
\def\GG{ {\bf G} }

\def\ZZ{ {\bf Z} }
\def\VV{ {\bf V} }
\def\UU{ {\bf U} }
\def\TI { \left({\rm T}\ii(\h)\right)}

\newtheorem{teo}{Theorem}[section]
\newtheorem{prop}[teo]{Proposition}
\newtheorem{lem}[teo]{Lemma}
\newtheorem{coro}[teo]{Corollary}
\newtheorem{defi}[teo]{Definition}
\theoremstyle{definition}
\newtheorem{rem}[teo]{Remark}
\newtheorem{ejem}[teo]{Example}

\title{The set of partial isometries as a quotient Finsler space}
\author{E. Andruchow\footnote{{\sc  {Instituto Argentino de Matem\'atica, `Alberto P. Calder\'on', CONICET, Saavedra 15 3er. piso,
(1083) Buenos Aires, Argentina }} and {{\sc Universidad Nacional de General Sarmiento, J.M. Gutierrez 1150, (1613) Los Polvorines, Argentina}} e-mail: eandruch@campus.ungs.edu.ar}}
\begin{document}

\maketitle 

\begin{abstract}
A known general  program, designed to endow the quotient space $\u_\a / \u_\b$ of the unitary groups $\u_\a$, $\u_\b$ of the C$^*$ algebras $\b\subset\a$ with an invariant Finsler metric,  is applied to obtain a  metric for the space $\ii(\h)$ of partial isometries of a Hilbert space $\h$. $\ii(\h)$  is a quotient of the unitary group of $\b(\h)\times\b(\h)$, where $\b(\h)$ is the algebra of bounded linear operators in $\h$. Under this program, the solution of a linear best approximation problem leads to the computation of minimal geodesics in the quotient space. We find solutions of this best approximation problem, and study properties of the minimal geodesics obtained.  \end{abstract}

\bigskip



\section{Introduction}
Let $\h$ be a complex Hilbert space, a {\it partial isometry} $V$ in $\h$ is an operator which is an isometry $V:\s_i\to\s_f$ between two closed subspaces $\s_i, \s_f\subset\h$ (called {\it initial} and {\it final} subspaces of $V$, respectively), and is zero on $\s_i^\perp$. Algebraically, this is equivalent to $VV^*V=V$, and in this case $V^*V$ and $VV^*$ are the orthogonal projections onto the spaces $\s_i$ and $\s_f$. Denote by 
$$
\ii(\h)=\{\hbox{partial isometries in }\h\}.
$$
The geometry of this set was thoroughly studied. Starting with Halmos and Mc Laughlin \cite{halmos mclaughlin}, who characterized the connected components.  Later on, other papers appeared studying geometric or topological aspects of the set of partial isometries, for instance: \cite{kovarik}, \cite{mbekhta1}, \cite{mbekhta2}, \cite{ac illinois}, \cite{lincei}, \cite{metric vna}, \cite{ch1}, \cite{ch2}.

Perhaps the main feature of $\ii(\h)$ is the left action of the group $\u(\h)\times \u(\h)$, where $\u(\h)$ denotes the unitary group of $\h$:
\begin{equation}\label{accion}
(U,W)\cdot V=UVW^* \ , \ \ (U,W)\in\u(\h)\times\u(\h) , \ V\in\ii(\h).
\end{equation}

The purpose of this note is to apply the program  by C. Dur\'an, L. Mata-Lorenzo and  L. Recht \cite{duranmatarecht}, devised for the study  of curves of minimal length in quotient  spaces of the group of unitary elements in a C$^*$-algebra, to the space $\ii(\h)$. 
Indeed, $\ii(\h)$ is a quotient of the unitary group of the C$^*$-algebra $\b(\h)\times\b(\h)$. The program in \cite{duranmatarecht} proceeds (roughly) as follows. If $M$ admits the transitive action of the unitary group $\u_\a$ of the C$^*$-algebra $\a$, then $M$ can be regarded as a quotient $\u_\a / G$ for $G$ a Banach-Lie subgroup of $\u_\a$. The program requires that $G$ be the unitary group of a unital sub-C$^*$-algebra $\b\subset\a$ (though this requirement can be sometimes relaxed or bypassed, as it is done here). Therefore, the tangent spaces ${\rm T}M$ are naturally isomorphic to the quotient of the (real) Banach spaces $\a_h / \b_h$, where $\b_h\subset\a_h$ denote the  sets of selfadjoint elements of $\b$ and $\a$, respectively. Using this isomorphism, Dur\'an, Mata-Lorenzo and Recht \cite{duranmatarecht} endowed ${\rm T} M$ with quotient metric of $\a_h / \b_h$. The metric, by design,  is invariant under the action of $\u_\a$ on $M$.  Therefore, tangent vectors can be lifted to selfadjoint elements in $\a$: the norm of such a vector is given by the infimum of the norms (measured in $\a$) of all possible liftings. Their main result states that if $m\in M$ and $v\in({\rm T}M)_m$ are given, and one can find a lifting $x_0\in\a_h$ of $v$, whose norm $\|x_0\|$ attains the infimum of all possible liftings of $v$, then the curve obtained as the uniparametric subgroup $e^{itx_0}$ acting on $m$, which at $t=0$ passes through $m$ with velocity $v$, has minimal length for this metric, at least for time $|t|\le\frac{\pi}{2\|x_0\|}$. Such liftings $x_0$ are called {\it minimal liftings},  their existence is not guaranteed, and even when they do exist, their characterization is an interesting problem, even in the case of finite dimensional algebras (i.e., matrix algebras): see for instance \cite{laa}, \cite{bottazzivarela}, \cite{kloboukvarela}. This problem is also related with non-commutative C$^*$-metrics and Leibniz seminorms \cite{rieffel}. An important background to the present work are the papers by E. Chiumiento \cite{ch1}, \cite{ch2}. In these papers, quotient metrics and minimal liftings are studied in the orbits of partial isometries under the action of the so called restricted groups of unitaries (i.e., unitaries which are of the form $1+K$, for $K$ in an operator ideal). 

Therefore, in dealing with particular examples, as is the case here, the focus is on the computation of such minimal liftings. 

Another antecedent of this aproach can be found in \cite{isometrias}, where the space of isometries was studied, though not with the quotient norm considered here

The contents of this note are the following. In Section 2 we state the basic facts on the space $\ii(\h)$ and the action of $\u(\h)\times\u(\h)$. In Section 3 we present an embedding of $\ii(\h)$ into the manifold of selfadjoint elements $\epsilon$ of $\h\times\h$ which satisfy $\epsilon^3=\epsilon$, i.e., are operators of the form $\epsilon=E_+-E_-$, with $E_+, E_-$ mutually orthogonal projections. In Section 4 we recall from \cite{duranmatarecht} the program of Dur\'an, Mata-Lorenzo and Recht. Following these ideas, in Section 5 we introduce  Finsler metrics in $\ii(\h)$ and in the space of the operators $\epsilon$ described in Section 3; with these metrics, the embedding of Section 3 is isometric. We show that the Finsler norm in ${\rm T}\ii(\h)$ is equivalent to the usual operator norm in $\b(\h)$.  We also prove the main theorem of this note: that curves obtained by the method of \cite{duranmatarecht} are not only minimal in $\ii(\h)$, but also minimal in the bigger manilfold of the operators $\epsilon$ of Sectin 3. In Section 6 we consider the initial and final projections maps
$$
\alpha(V)=V^*V \ , \ \ \omega(V)=VV^*.
$$
It is shown that if the set of projections is considered with its natural Finsler metric (see \cite{cpr}), both maps are distance decreasing.

\section{Preliminaries}
Let us recall the basic facts of the space $\ii(\h)$ and the action (\ref{accion}):
\begin{rem}
Let $V, V_0, V_1, V_2\in\ii(\h)$.
\noindent

\begin{enumerate}
\item
The connected components of $\ii(\h)$ are parametrized by three non negative integers $\le +\infty$: 
$$
r(V)=\dim R(V) , \ n(V)=\dim N(V) \hbox{ and } r^\perp(V)=\dim R(V)^\perp.
$$
Namely, $V_1, V_2\in\ii(\h)$ belong to the same connected component if and only if $r(V_1)=r(V_2)$, $n(V_1)=n(V_2)$ and $r^\perp(V_1)=r^\perp(V_2)$ (see \cite{halmos mclaughlin}). 
\item
If $\|V_1-V_2\|<1$, then $V_1$ and $V_2$ lie in the same connected component of $\ii(\h)$ see \cite{halmos mclaughlin}.  

\item
These components coincide with the orbits of the action (\ref{accion}): $V_1, V_2$ lie in the same component if and only if there exist $U,W\in\u(\h)$ such that $UV_1W^*=V_2$.
\item
More recently, in \cite{ac illinois} we considered the set of partial isometries of a C$^*$-algebra, as a homogeneous manifold. Each connected component / orbit, is a $C^\infty$ complemented  submanifold of the algebra. Back to the case when the algebra is $\b(\h)$, if $\ii(\h)_{V_0}$ denotes the connected component of $V_0$, then the map
\begin{equation}\label{el mapa}
\pi_{V_0}:\u(\h)\times\u(\h)\to\ii(\h)_{V_0} , \ \pi_{V_0}(U,W)=UV_0W^*
\end{equation}
is a $C^\infty$-submersion (see \cite{ac illinois}). Note then that the whole space $\ii(\h)$ is a discrete union of complemented submanifolds (any two different components lying at distance of at least $1$), and therefore is  itself a complemented submanifold of $\b(\h)$. 
\item
Given $V_0$, the subgroup of elements in $\u(\h)\times\u(\h)$ which fix the element $V_0$, usually called the {\it isotropy subgroup} of $V_0$, is given by
\begin{equation}\label{grupo i}
{\rm I}_{V_0}=\{(G,H)\in\u(\h)\times\u(\h): GV_0=V_0H\}.
\end{equation}
It is a $C^\infty$ Banach-Lie group, whose Banach-Lie algebra is 
\begin{equation}\label{algebra i}
\iota_{V_0}=\{ (iX,iY)\in\b(\h)\times\b(\h): X^*=X, Y^*=Y \hbox{ and } XV_0=V_0Y\}.
\end{equation}

\end{enumerate}

\end{rem}

$V\in\ii(\h)$ is called {\it balanced} if $n(V)=r^\perp(V)$. This type of partial isometries comprise full connected components of $\ii(\h)$. Note that $V\in\ii(\h)$ is balanced if and only if there exists an orthogonal projection $P$ and unitaries $U,W$  such that $V=UPW^*$. Indeed,  if $V=UPW^*$, then $V$ and $P$ lie in the same connected component, and therefore, by Halmos-McLaughlin's characterization, $n(V)=n(P)=r^\perp(P)=r^\perp(V)$. The converse statement is clear.

Non-unitary isometries are examples of non balanced partial isometries.

The purpose of this note, is to introduce a natural invariant Finsler metric in $\ii(\h)$. That is, a metric  $|\ \ |_V$ in each tangent  space $\TI_V$, which is invariant under the action of $\u(\h)\times\u(\h)$: if $\v\in\TI_V$ and $U,W\in\u(\h)$, then
$$
|\v|_V=|U\v W^*|_{UVW^*}.
$$
Here, note that the action is linear: for fixed $U,W\in\u(\h)$, the map $V\mapsto UVW^*$ is the restriction of a global bounded linear map $X\mapsto UX W^*$.

It will be useful to recall the form of the tangent spaces of $\ii(\h)$. The fact that the map (\ref{el mapa}) is a submersion, implies that its tangent maps are surjective. Then
$$
({\rm T}\pi_{V_0})_{(1,1)}: (T\u(\h)\times\u(\h))_{(1,1)}=\{(iX,iY): X^*=X, Y^*=Y\}\to \TI_{V_0}, 
$$
$$
({\rm T}\pi_{V_0})_{(1,1)}((iX,iY)=iXV_0-iV_0Y
$$
is surjective, and
\begin{equation}\label{tangente I}
\TI_{V_0}=\{iXV_0-iV_0Y: X^*=X, Y^*=Y\}.
\end{equation}
The metric that will be considered in ${\rm T}\ii(\h)$ is a {\it quotient metric}, using the homogeneous structure of $\ii(\h)$ (as a quotient of the group $\u(\h)\times\u(\h)$), following the program outlined in the seminal paper by Dur\'an, Mata-Lorenzo and Recht \cite{duranmatarecht}. We shall describe it in Section \ref{programa}.

Following this program, one can compute curves of minimal length ({\it metric geodesics}) of this Finsler metric, by finding {\it minimal liftings} of tangent vectors in ${\rm T}\ii(\h)$. 

\section{$2\times 2$ model for $\ii(\h)$}

Given $V\in\ii(\h)$, consider $\epsilon_V\in\b(\h\times\h)$ given by
$$
\epsilon_V=\left(\begin{array}{cc} 0 & V \\ V^* & 0 \end{array} \right).
$$
Note that $\epsilon_V^*=\epsilon_V$,
$$
\epsilon_V^2=\left(\begin{array}{cc} VV^* & 0 \\ 0 & V^*V \end{array} \right),
$$
where $V^*V$ and $VV^*$ are the initial and final projections of $V$, and that
$$
\epsilon_V^3=\left(\begin{array}{cc} 0 & VV^*V  \\ V^*VV^* & 0 \end{array} \right)=\left(\begin{array}{cc}  0 & V \\ V^* & 0 \end{array} \right)=\epsilon_V
$$
It follows that $\epsilon=\epsilon_V$ is a selfadjoint root of the polynomial $\xx^3-\xx$, and therefore has a simple spectral decomposition of the form
$$
\epsilon_V= 0\cdot E_0+ 1\cdot E_+-1\cdot E_-=E_+-E_-,
$$
with 
$$
E_+=\frac12\{\epsilon_V^2+\epsilon_V\}, \ E_-=\frac12\{\epsilon_V^2-\epsilon_V\} \ \hbox{ and } 
E_0=1-\epsilon_V^2,
$$
the mutually orthogonal spectral projections of $\epsilon_V$.
\begin{rem}
Consider $\epsilon=\epsilon^*$ with $\epsilon^3=\epsilon$.
The unitary orbit of $\epsilon$, under the inner action of the unitary group of $\h\times\h$,
$$
\o_\epsilon=\{\UU\ \epsilon\ \UU^*: \UU=\left(\begin{array}{cc} U_{11} & U_{12} \\ U_{21} & U_{22} \end{array} \right) \in\u(\h\times\h)\},
$$
is a complemented $C^\infty$-submanifold of $\b(\h\times\h)$, and a homogeneous space of the unitary group $\u(\h\times\h)$ (see \cite{cpr systems}). The isotropy group of $\epsilon$ is given by
$$
{\rm I}_{\epsilon}=\{\GG=\left(\begin{array}{cc} G_{11} & G_{12} \\ G_{21} & G_{22} \end{array} \right) \in\u(\h\times\h): \GG \ \epsilon=\epsilon \ \GG\}.
$$
The Banach-Lie algebra, in the special case where $\epsilon=\epsilon_V$ for some $V\in\ii(\h)$, is
$$
\iota_{\epsilon_V}=\Big\{i\XX=i\left(\begin{array}{cc} X_{11} & X_{12} \\ X_{12}^* & X_{22} \end{array} \right) \in\b(\h\times\h): X_{ii}^*=X_{ii}  \hbox{ and }\left\{\begin{array}{l} X_{12}V^*=VX_{12}^*, \\  X_{12}^*V=V^*X_{12}, \\  X_{11}V=VX_{22}\end{array}\right. \Big\}.
$$
\end{rem}
If we restrict the above inner action (on $\epsilon_V$) to the diagonal subgroup
$$
\Delta=\Big\{\left( \begin{array}{cc} U & 0 \\ 0 & W \end{array} \right) : U, W\in\u(\h)\Big\}\subset \u(\h\times\h),
$$
we obtain a copy of the connected component (orbit) of $V$:
\begin{prop}
Let $V\in\ii(\h)$, then
$$
\Delta\cdot \epsilon_V=\Big\{\left( \begin{array}{cc} 0 & UVW^* \\ (UVW^*)^* &  0 \end{array}\right): U,W\in\u(\h)\Big\}=\{\epsilon_{UVW^*}: U,W\in\u(\h)\}\simeq \ii_V.
$$
\end{prop}
\begin{proof}
It is a straightforward computation:
$$
\left( \begin{array}{cc} U & 0 \\ 0 & W \end{array} \right) \left( \begin{array}{cc}
0 & V \\ V^* & 0 \end{array}\right)\left( \begin{array}{cc} U^* & 0 \\ 0 & W^* \end{array} \right)=\left( \begin{array}{cc}
0 & UVW^* \\ WV^*U^* & 0 \end{array}\right)
$$
\end{proof}
Note that 
$$
(T\Delta\cdot\epsilon_V)_{\epsilon_V}=\Big\{\left( \begin{array}{cc} 0 & iXV-iVY \\ iYV^*
-iV^*X &  0 \end{array}\right): X^*=X, Y^*=Y\Big\}.
$$
\section{The program of Dur\'an, Mata-Lorenzo and Recht}\label{programa}

Let us briefly describe the context and main result of \cite{duranmatarecht}. Let $\b\subset\a$ be unital C$^*$-algebras (with the same unit). Denote by $\u_\b$ and $\u_\a$ the unitary groups of $\b$ and $\a$, respectively. In \cite{duranmatarecht} a metric was introduced  in the homogeneous (quotient) space $M=\u_\a \ / \ \u_\b$. If $u\in\u_\a$, let $[u]\in M$ be the class of $u$ in the quotient space. The Banach-Lie algebras of $\u_\a$ and $\u_\b$ are, respectively
$$
\mathfrak{u}_\a=\{ix: x^*=x\in\a\} , \ \mathfrak{u}_\b=\{iy: y^*=y\in\a\}.
$$
For $[u]\in M$, let
$$
\pi_{[u]}:\u_\a\to M , \ \pi_{[u]}(w)=[uw].
$$

The tangent space
of $M$ at $[u]$  naturally identifies with the quotient of the Lie algebras of $\u_\a$ and $\u_\b$:
$$
({\rm }T M)_{[u]}\simeq \mathfrak{u}_\a \ / \ \mathfrak{u}_\b,
$$
since $({\rm T}\pi_{[u]})_{[1]}:\mathfrak{u}_\a\to({\rm T}M)_{[u]}$ is an epimorphism with nullspace $\mathfrak{u}_\b$.
In this tangent space they define the natural metric:
\begin{defi}
If $\vv=ix+\mathfrak{u}_\b \in (TM)_{[u]}$
\begin{equation}\label{norma en 1}
|\vv|_{[u]}:=\inf\{\|x+y\|: y^*=y\in\b\},
\end{equation}
i.e., the usual metric in the quotient of (real) Banach spaces $\mathfrak{u}_\a \ / \ \mathfrak{u}_\b$. 

\end{defi}
\begin{defi}
Let $\vv\in(TM)_{[u]}$. An element $x_0=x_0^*\in\a$ is a {\it minimal lifting} of $\vv$ if 
$$
\vv=i x_0+\mathfrak{u}_\b
$$
and
$$
\|x_0\|=\inf\{\|x_0+y\|: y=y^*\in\b\}=|\vv|_{[u]}.
$$
That is, $x_0$ attains the norm of the class in the quotient norm.
\end{defi}

In general, minimal liftings may not exist (see for instance the paper \cite{bottazzivarela} for an interesting example). However, when they do exist, they provide curves of minimal length in $M$:
\begin{teo}\label{teo dmr} {\rm (Dur\'an, Mata-Lorenzo, Recht \cite{duranmatarecht})}

Let $[u]\in M$ and $\vv\in(TM)_{[u]}$. Suppose that $\vv$ has a minimal lifting $x_0$. Then the curve
$$
\delta(t)=[e^{itx_0}u]
$$ 
which satisfies the initial conditions
$$
\delta(0)=[u] \ \hbox{ and } \ \dot{\delta}(0)=\vv,
$$
has minimal length along its path for $|t|\le \frac{\pi}{2|\vv|_{[u]}}=\frac{\pi}{2\|x_0\|}$.
\end{teo}
Here, by {\it minimal length along its path} at the given interval of $t$, means that if $[t_0,t_1]$ is a subinterval of $[-\frac{\pi}{2|\vv|_{[u]}},\frac{\pi}{2|\vv|_{[u]}}]$, and $\gamma(t)$  ($t\in I$) is an arbitrary smooth curve in $M$  joining $\delta(t_0)$ and $\delta(t_1)$, then
$$
\ell(\delta|_{[t_0,t_1]})=\int_{t_0}^{t_1} |\dot{\delta}(t)|_{\delta(t)} dt\le   \int_I |\dot{\gamma}(t)|_{\gamma(t)} dt =\ell(\gamma).
$$
\begin{ejem}\label{34}
An example where minimal liftings exist at every tangent vector (at every point), occurs when both $\a$ and $\b$ are von Neumann algebras. For instance, if $V\in\ii(\h)$, then 
$$
M=\o_{\epsilon_V}\simeq \u(\h\times\h) \ / \ \u(\h\times\h)\cap \{\epsilon_V\}'
$$
is such an example. Indeed, $\u(\h\times\h)\cap \{\epsilon_V\}'$ is the unitary group of the von Neumann algebra $\{\epsilon_V\}'\subset \b(\h\times\h)$. 

\end{ejem}
\section{Finsler metric and minimal curves in $\ii(\h)$}
In this section we show that minimal liftings of the homogeneous space 
$$
\o_{\epsilon_V}\simeq \u(\h) \ / \u(\h)\cap\{\epsilon_V\}'
$$
induce in a simple manner minimal liftings in $\ii(\h)$, or  more precisely, in the connected component $\ii(\h)_V \simeq \u(\h)\times\u(\h) \ / \ {\rm I}_V$ of $V$ in $\ii(\h)$. 

Note that ${\rm I}_V$ is not the unitary group of a selfadjoint algebra in a straightforward fashion. We shall use the $2\times 2$ model for $\ii(\h)$, in order to be able to obtain minimal liftings, and as a byproduct, a stronger minimality result. Namely, that the metric geodesics obtained are not only (locally) minimal in $\ii(\h)$, but also in the ambient manifold $\o_{\epsilon_V}$ (regarding $\ii(\h)_V$ as a subset of $\o_{\epsilon_V}$ via the isometric embedding $\ii(\h)_V\hookrightarrow \o_{\epsilon_V}$, $V\mapsto\epsilon_V$).

Following the program in \cite{duranmatarecht}, we define the following Finsler metrics in $\o_{\epsilon_V}$ and in $\ii(\h)$. 
If $\epsilon^3=\epsilon^*=\epsilon$, and $\v\in(T\o_\epsilon)_\epsilon$, we put
\begin{equation}\label{metrica epsilonV}
|\v|_\epsilon=\inf\{\|\XX\|: ({\rm T}\pi_\epsilon)_1(\XX)=\v\}=\inf\{\|\XX+\ZZ\|: \ZZ\in\iota_\epsilon\}
\end{equation}
If $V\in\ii(\h)$ and $\v\in\TI_V$, we put
\begin{equation}\label{metrica V}
|\v|_V=\inf\{\|(A,B)\|: A,B\in\b(\h), A^*=A, B^*=B \hbox{ and } iAV-iVB=\v\},
\end{equation}
where as is usual $\|(A,B)\|=\max\{\|A\|,\|B\|\}$ (i.e., the C$^*$-norm in $\b(\h)\times\b(\h)$).
\begin{rem}
Clearly, definitions (\ref{metrica epsilonV}) and (\ref{metrica V}) make
$$
\ii(\h)\to \Delta\cdot\epsilon_V \ , \ V\mapsto \epsilon_V
$$
an isometric diffeomorphism.
\end{rem}

Before we proceed, we state the following results, which are elementary and known, and will be used thoroughly. The first fact, is that if one deals with $n\times n$ (block) operator matrices, the diagonal map
\begin{equation}\label{pinching}
E\Big( \left( \begin{array}{cccc} A_{11} & A_{12} & \dots & A_{1n} \\  A_{21} & A_{22} & \dots & A_{2n} \\ \dots & \dots & \dots & \dots \\ A_{n1} & A_{n2} & \dots & A_{nn}\end{array}\right) \Big)=\left( \begin{array}{cccc} A_{11} & 0 & \dots & 0 \\  0 & A_{22} & \dots & 0 \\ \dots & \dots & \dots & \dots \\ 0 & 0 & \dots & A_{nn}\end{array}\right)
\end{equation}
is positive, linear and contractive. The second fact is the following:
\begin{lem}\label{area 51}
Let $A, P$  in $\b(\h)$, $P$ an orthogonal projection. Regard $A$ as a $2\times 2$ matrix in terms of $P$:
$$
A=\left( \begin{array}{cc} A_{11} & A_{12} \\ A_{21} & A_{22} \end{array}\right) \begin{array}{l} R(P) \\ N(P) \end{array}.
$$
Then
$$
\Big\| \left( \begin{array}{cc} 0 & A_{12} \\ A_{21} & 0 \end{array}\right)\Big\| \le \Big\|\left( \begin{array}{cc} A_{11} & A_{12} \\ A_{21} & A_{22} \end{array}\right)\Big\|.
$$
\end{lem}
\begin{proof}
Note that
$$
\|A\|^2=\|A^*A\|=\Big\|\left( \begin{array}{ccc} A_{11}^*A_{11}+A_{21}^*A_{21} &  A_{11}^*A_{12}+A_{21}^*A_{22}  \\ A_{12}^*A_{11}+A_{22}^*A_{21} &  A_{12}^*A_{12}+A_{22}^*A_{22} \end{array}\right)\Big\|\ge \|E(A^*A)\|,
$$
where $E$ is the linear map given in (\ref{pinching}),
$$
\|E(A^*A)\|=\Big\|\left( \begin{array}{cc} A_{11}^*A_{11}+A_{21}^*A_{21} & 0 \\ 0 & A_{12}^*A_{12}+A_{22}^*A_{22} \end{array}\right)\Big\|
$$
Since clearly 
$$
\left( \begin{array}{cc} A_{11}^*A_{11}+A_{21}^*A_{21} & 0 \\ 0 & A_{12}^*A_{12}+A_{22}^*A_{22} \end{array}\right)\ge \left( \begin{array}{cc} A_{21}^*A_{21} & 0 \\ 0 & A_{12}^*A_{12} \end{array}\right)
$$
we get
$$
\|A\|^2\ge \Big\|\left( \begin{array}{cc} A_{21}^*A_{21} & 0 \\ 0 & A_{12}^*A_{12} \end{array}\right)\Big\|=\Big\| \left( \begin{array}{cc} 0 & A_{21}^* \\ A_{12}^* & 0 \end{array}\right)\left( \begin{array}{cc} 0 & A_{12} \\ A_{21} & 0 \end{array}\right)
\Big\|=\Big\|\left( \begin{array}{cc} 0 & A_{12} \\ A_{21} & 0 \end{array}\right)\Big\|^2
$$
\end{proof}
Our next goal is to compare $|\v|_V$ with the usual operator norm $\|\v\|$ of $\v\in\TI_V$, for $V\in\ii(\h)$. In fact, we shall see that for each fixed $V\in\ii(\h)$,  $|\ \ |_V$ and $\|\ \|$ are equivalent in $\left({\rm T}\ii(\h)\right)_V$. To do this task, we shall need a classical result by M.C. Krein \cite{krein}, known as the extension problem for symmetric transformations (see also the excellent text \cite{riesz nagy}, Section 125, or also \cite{davis et al}, \cite{parrott} for more nuanced developements on this subject).
We state this result in the following remark, adapted to our particular problem:
\begin{rem}
Let 
$$
\left(\begin{array}{cc} A_{11} & A_{12} \\ A_{12}^* & *\end{array}\right)
$$
be an incomplete operator matrix, with $A_{11}^*=A_{11}$. Then there exist (non unique) selfadjoint completions 
$$
\AA=\left(\begin{array}{cc} A_{11} & A_{12} \\ A_{12}^* & A_{22}\end{array}\right)
$$
with $\|\AA\|=\Big\|\left(\begin{array}{c} A_{11} \\ A_{12}^*\end{array}\right)\Big\|=\|\left(\begin{array}{lr} A_{11} & A_{12}  \end{array}\right)\|$.
\end{rem}

\begin{teo}\label{comparacion}
Let $V\in\ii(\h)$ and $\v\in\TI_V$. Then 
$$
|\v|_V\le \|\v\|\le 2 |\v|_V
$$
\end{teo}
\begin{proof}
Let us first consider the case of a balanced partial isometry, $n(V)=r^\perp(V)$. Let $U,W$ be unitaries such that $UP_0W^*=V$, for an orthogonal projection $P_0$. Clearly, pulling  back $\v$ with the left action of the pair $(U,W)$, it suffices to reason in the case $V=P_0$. To this effect, note that the action of $(U,W)$ is isometric both for the Finsler norm $|\ \ |_V$ and the operator norm $\|\ \|$. Let $(A,B)$ be a lifting for $\v$, i.e., $A^*=A$, $B^*=B$ and $iAP_0-iP_0B=\v$. Note that, in matrix form in terms of $P_0$,
$$
\v=i\left(\begin{array}{cc} A_{11} & A_{12} \\ A_{12}^* & A_{22}\end{array}\right)
\left(\begin{array}{cc} 1 & 0 \\ 0  & 0\end{array}\right)-i \left(\begin{array}{cc} 1 & 0 \\ 0 & 0\end{array}\right)\left(\begin{array}{cc} B_{11} & B_{12} \\ B_{12}^* & B_{22}\end{array}\right)=i\left(\begin{array}{cc} A_{11}-B_{11} & B_{12} \\ A_{12}^* & 0 \end{array}\right)
$$ 
Let us alter the lifting $(A,B)$. Put
$$
B_0=\left(\begin{array}{cc} 0 & B_{12} \\ B_{12}^* & 0\end{array}\right)
$$
and 
$$
A_0=\left(\begin{array}{cc} A_{11}-B_{11} & A_{12} \\ A_{12}^* & Z\end{array}\right)
$$
where $Z=Z^*$ is such that $A_0$ is a solution of Krein's extension problem for the symmetric incomplete matrix
$$
\left(\begin{array}{cc} A_{11}-B_{11} & A_{12} \\ A_{12}^* & *  \end{array}\right)
$$
Straightforward computations show that $(A_0,B_0)$ is also a lifting of $\v$: $A_0^*=A_0$, $B_0^*=B_0$ and  $iA_0P_0-iP_0B_0=\v$.
Then 
$$
\|A_0\|=\Big\|\left(\begin{array}{c} A_{11}-B_{11}  \\ A_{12}^*\end{array}\right)\Big\|=\Big\| \left(\begin{array}{cc} A_{11}- B_{11} & 0 \\ A_{12}^* & 0\end{array}\right)\Big\|=\Big\|\left(\begin{array}{cc} A_{11} - B_{11} & B_{12} \\ A_{12}^* & 0\end{array}\right)\left(\begin{array}{cc} 1 & 0 \\ 0 & 0\end{array}\right)\Big\|
$$
$$
=\|\v P_0\|\le\|\v\|.
$$
On the other hand
$$
\Big\|\left(\begin{array}{cc} 0 & B_{12} \\ B_{12}^* & 0\end{array}\right)\Big\|=\Big\|\left(\begin{array}{cc} 0 & B_{12} \\ 0 & 0\end{array}\right)\Big\|=\Big\|\left(\begin{array}{cc} A_{11}-B_{11} & B_{12} \\ A_{12}^* & 0\end{array}\right)\left(\begin{array}{cc} 0 & 0 \\ 0 & 1\end{array}\right)\Big\|=\|\v P_0^\perp\|\le \|\v\|.
$$
Thus, we have found a lifting $(A_0,B_0)$ of $\v$ such that 
$$
\|(A_0,B_0)\|=\max\{\|A_0\|,\|B_0\|\}\le \|\v\|.
$$
It follows that $|\v|_{P_0}\le\|\v\|$.

Let us now consider the general case. Consider the Hilbert space $\h\times\h$.
Note that $V\oplus 0$ in $\h\times\h$ defined as $V(\xi,\eta)=(V\xi,0)$ is a partial isometry with $n(V\oplus 0)=r^\perp(V\oplus 0)=+\infty$. Similarly, if $\v\in\TI_V$, then $\v\oplus 0\in\left({\rm T} \ii(\h\times\h)\right)_{V\oplus 0}$. Note that
$$
|\v\oplus 0|_{V\oplus 0}\le |\v|_V.
$$
Indeed, any lifing $(A,B)$ of $\v$, provides a lifting $(A\oplus 0,B\oplus 0)$ of $\v\oplus 0$. On the other hand, if
$$
\left(\begin{array}{cc} A_{11} & A_{12} \\ A_{12}^* & A_{22}\end{array} \right) \ , \ 
\left(\begin{array}{cc} B_{11} & B_{12} \\ B_{12}^* & B_{22}\end{array} \right)
$$ 
is a lifting of $\v\oplus 0$. i.e.,
$$
\v\oplus 0=\left(\begin{array}{cc} \v & 0 \\ 0 & 0 \end{array} \right)=i\left(\begin{array}{cc} A_{11}V-VB_{11} & -VB_{12} \\ A_{12}^*V & 0\end{array} \right),
$$
then, in particular $(A_{11},B_{11})$ is a lifting for $\v$.
Since 
$$
\|A_{11}\|\le\Big\| \left(\begin{array}{cc} A_{11} & A_{12} \\ A_{12}^* & A_{22}\end{array} \right)\Big\| \ \hbox{ and } \ \|B_{11}\|\le \Big\|\left(\begin{array}{cc} B_{11} & B_{12} \\ B_{12}^* & B_{22}\end{array} \right)\Big\|,
$$
it follows that $|\v|_V\le |\v\oplus 0|_{V\oplus 0}$.
Then, by the first case,
$$
\|\v\|=\|\v\oplus 0\|\le|\v\oplus 0|_{V\oplus 0}=|\v|_V.
$$

The other inequality is trivial:
$$
\|\v\|=\|AV-VB\|\le \|AV\|+\|VB\|\le \|A\|+\|B\|\le 2 \max\{\|A\|,\|B\|\},
$$
for any lifting $(A,B)$ of $\v$.

\end{proof}
\begin{rem}
The inequality $|\v|_V\le\|\v\|$ may be strict. The problem of finding completions of $2\times 2$ matrices with minimal norm has been studied for non-selfadjoint matrix operators (see for instance \cite{davis et al}). Namely, applied in our context and following the notations of the above theorem, given the incomplete (non-selfadjoint) matrix operator
$$
\left(\begin{array}{cc} A_{11}-B_{11} & B_{12} \\ A_{12}^* & * \end{array}\right)
$$
There exists a completion $C=\left(\begin{array}{cc} A_{11}-B_{11} & B_{12} \\ A_{12}^* & Y \end{array}\right)$ with minimal norm, that is
$$
\|C\|=\max\{\|\left( \begin{array}{lr} A_{11}-B_{11} & B_{12} \end{array}\right)\|, \Big\|\left(\begin{array}{c} A_{11}-B_{11} \\ A_{12}^*\end{array}\right)\Big\|\}.
$$
The row $\left( \begin{array}{lr} A_{11}-B_{11} & B_{12} \end{array}\right)$ is the first row of the incomplete matrix $\left( \begin{array}{cc} A_{11}-B_{11} & B_{12} \\ B_{12}^* & *  \end{array}\right)$
which can be completed with minimal norm to the selfadjoint operator
$
\left( \begin{array}{cc} A_{11}-B_{11} & B_{12} \\ B_{12}^* & Z'  \end{array}\right),
$
with 
$$
\|\left( \begin{array}{lr} A_{11}-B_{11} & B_{12} \end{array}\right)\|=\Big\|\left( \begin{array}{cc} A_{11}-B_{11} & B_{12} \\ B_{12}^* & Z'  \end{array}\right)\Big\|.
$$
By Lemma \ref{area 51}, 
$$
\Big\|\left( \begin{array}{cc} A_{11}-B_{11} & B_{12} \\ B_{12}^* & Z'  \end{array}\right)\Big\|\ge \Big\|\left( \begin{array}{cc} 0 & B_{12} \\ B_{12}^* & 0  \end{array}\right)\Big\|.
$$
Similarly, reasoning with the first column $\left(\begin{array}{c} A_{11}-B_{11} \\ A_{12}^*\end{array}\right)$, we get that there is a selfadjoint completion $\left(\begin{array}{cc} A_{11}-B_{11} & A_{12}\\ A_{12}^* & Z\end{array}\right)$ such that
$$
\Big\|\left(\begin{array}{c} A_{11}-B_{11} \\ A_{12}^*\end{array}\right)\Big\|=\Big\|\left(\begin{array}{cc} A_{11}-B_{11} & A_{12} \\ A_{12}^* & Z\end{array}\right)\Big\|.
$$
It follows that 
$$
\|C\|\ge \max\{\Big\|\left( \begin{array}{cc} 0 & B_{12} \\ B_{12}^* & 0  \end{array}\right)\Big\|, \Big\|\left(\begin{array}{cc} A_{11}-B_{11} & A_{12} \\ A_{12}^* & Z\end{array}\right)\Big\|\},
$$
which is the norm of a lifting $(A_0,B_0)$ of $\v$ (as in the first part of the proof of the above Theorem). That is, $\|C\|\ge |\v|_V$. Now, $C$ and $-i\v$ are both completions of the same incomplete (non-selfadjoint) matrix. Since $C$ has minimal norm among these completions, one has $\|C\|\le\|\v\|$. Moreover, it is known that, in general, putting $0$ in the $2,2$ place is not the optimal solution (see \cite{davis et al}, \cite{parrott}): there are examples where $\|C\|<\|\v\|$. Then, for such $\v$, we have $|\v|_V<\|\v\|$.
\end{rem}

In order to establish the existence of minimal liftings in $\ii(\h)$, we need the next lemma.
\begin{lem}
Let $V\in\ii(\h)$ and  $\AA=\left( \begin{array}{cc} A_1 & 0 \\ 0 & A_2 \end{array}\right)$ be selfadjoint in $\h\times\h$. Let
$$
\XX=\left( \begin{array}{cc} X_{11} & X_{12} \\ X_{12}^* & X_{22} \end{array}\right)\in  \{ \epsilon_V\}' , \ \XX^*=\XX^*,
$$
such that 
$$
\| \AA+\XX\|\le \|\AA+\YY\| \hbox{ for all } \YY^*=\YY\in\{\epsilon_V\}'
$$
(which exists, recall Example \ref{34}). 
Then $\XX_0:=\left( \begin{array}{cc} X_{11} & 0 \\ 0 & X_{22} \end{array}\right)$ satisfies
\begin{itemize}
\item
$\XX_0$ commutes with $\epsilon_V$, in particular, $i \XX_0\in \iota_V$;
\item
$\|\AA+\XX_0\|\le \|\AA+\ZZ\|$ for all $i\ZZ\in\iota_V$.
\end{itemize}
\end{lem}
\begin{proof}
Clearly, $\XX_0^*=\XX_0$. The fact that $\XX$ commutes with $\epsilon_V$, means that
$$
\left\{\begin{array}{l} X_{12}V^*=VX_{12}^*, \\  X_{12}^*V=V^*X_{12}, \\  X_{11}V=VX_{22}\end{array}\right.
$$
and thus also $\XX_0$ commutes with $\epsilon_V$, in particular $X_{11}V=VX_{22}$,  which means that  $i\XX_0\in\iota_V$. By the same argument, it is  also  clear that if 
$\ZZ=\left( \begin{array}{cc} Z_{11} & Z_{12} \\ Z_{12}^* & Z_{22} \end{array}\right)$
commutes with $\epsilon_V$, then also $\ZZ_0=\left( \begin{array}{cc} Z_{11} & 0 \\ 0 & Z_{22} \end{array}\right)$ also commutes with $\epsilon_V$. Then
$$
\|\AA+\XX\|\le \|\AA+\ZZ_0\|.
$$
On the other hand, the linear map $E:\b(\h\times\h)\to \b(\h\times\h)$ given by
$$
E\big(\left( \begin{array}{cc} T_{11} & T_{12} \\ T_{21} & T_{22} \end{array}\right)\big)=\left( \begin{array}{cc} T_{11} & 0 \\ 0 & T_{22} \end{array}\right)
$$ 
is contractive. Then
\begin{equation}\label{E}
\|\AA+\XX_0\|=\|E(\AA+\XX)\|\le\|\AA+\XX\|\le\|\AA+\ZZ_0\|,
\end{equation}
which completes the proof.
\end{proof}
We shall call $\AA_0:=\AA+\XX_0$ a {\it diagonal minimal lifting}. Note that the above Lemma states that any vector tangent to $\Delta\cdot\epsilon_V$ has a diagonal minimal lifting $\AA_0$. 
Applying Theorem \ref{teo dmr} \cite{duranmatarecht} we get
\begin{teo}
Let $V\in\ii(\h)$ and $\v=iXV-iVY\in\TI_V$. Let $\VV=\left( \begin{array}{cc} 0 & \v \\ \v^* & ^0 \end{array}\right)\in(T\Delta\cdot\epsilon_V)_V$, and pick $\AA_0 $ a diagonal minimal lifting for $\VV$. Then the curve
$$
\delta(t)=e^{it\AA_0}\ \epsilon_V \ e^{-it\AA_0}
$$
which satisfies that $\delta(0)=\epsilon_V$ and $\dot{\delta}(0)=\VV$, has minimal length along its path in $\Delta \cdot \epsilon_V$, for $|t|\le \frac{\pi}{2|\VV|_V}$. Moreover, it also has minimal length among curves in bigger manifold $\o_{\epsilon_V}$, in the same time interval.
\end{teo}
\begin{proof}
The proof of the above Lemma, in fact shows that $\AA+\XX_0$ is a minimal lifting in the bigger quotient space.
\end{proof}
Therefore, if we consider the quotient left invariant metric in $\ii(\h)$, we obtain:
\begin{coro}\label{corolario 47}
Let $V\in\ii(\h)$ and $\v=iXV-iVY\in\TI_V$. Then there exist $X_0^*=X_0$, $Y_0^*=Y_0$ with $\v=iX_0V-iVY_0$, such that the curve
$$
\delta(t)=e^{itX_0}Ve^{-itY_0}
$$
which satisfies $\delta(0)=V$ and $\dot{\delta}(0)=\v$, has minimal length along its path in $\ii(\h)$, for $|t|\le\frac{\pi}{2|\v|_V}$.
\end{coro}

\section{Initial and final projections}
If $V\in\ii(\h)$, denote by $\alpha(V)=V^*V$ and $\omega(V)=VV^*$ the initial and final projections of $V$. Denote by $\p(\h)$ the space of (orthogonal) projections of $\b(\h)$. The space of projections of a C$^*$-algebra has been well studied, as a complemented submanifold of the algebra, and as an homogeneous space of the inner action of the unitary group of the algebra ($u\cdot p=upu^*$, if $u$ is unitary and $p$ is a projection). It has also been studied as a  Finsler metric space, where each tangent space is endowed with the usual norm of the algebra (see \cite{pr} and \cite{cpr}). In the specific case of the algebra $\b(\h)$, existence of minimal geodesics with given initial conditions or with given  endpoints, have been characterized (see the references above, or \cite{p-q} for the specific case of the algebra $\b(\h)$). 

Clearly, the maps
$$
\alpha:\ii(\h)\to \p(\h),\  \alpha(V)=V^*V\ \ \hbox{ and }\ \ \omega:\ii(\h)\to \p(\h),\  \omega(V)=VV^*
$$
are $C^\infty$. Let us show that, if $\p(\h)$ is given the above mentioned Finsler metric, i.e., the usual norm at every tangent space, and $\ii(\h)$ is considered with the quotient metric studied here, then both maps $\alpha$ and $\omega$ decrease distances.
This fact is based in Lemma \ref{area 51}.
As said above, we consider $\ii(\h)$ and $\p(\h)$ as metric spaces, with their given Finsler metrics. Recall how a Finsler metric in the tangent spaces induces a metric in the original space: if $M$ is a manifold with a Finsler metric $|\ \ |_m$ at $(TM)_m$ (for $m\in M$), then
$$
d_M(m_1,m_2)=\inf\{ \ell(\gamma): \gamma(t)\in M, t\in [a,b], \gamma \hbox{ is smooth }, \gamma(a)=m_1, \gamma(b)=m_2\},
$$
where
$$
\ell(\gamma)=\int_a^b |\dot{\gamma}(t)|_{\gamma(t)} d t.
$$

\begin{prop}\label{contracciones}
The maps 
$$
\alpha:\ii(\h)\to\p(\h) , \ \alpha(V)=V^*V
$$
and
$$
\omega:\ii(\h)\to\p(\h) , \ \omega(V)=VV^*
$$
are distance decreasing, i.e., if $V_1, V_2\in\ii(\h)$, $E_i=\alpha(V_i)$, $F_i=\omega(V_i)$, $i=1,2$, then
$$
d_{\p(\h)}(E_1,E_2)\le d_{\ii(\h)}(V_1,V_2) \ \ \ \hbox{ and }\ \ \ d_{\p(\h)}(F_1,F_2)\le d_{\ii(\h)}(V_1,V_2).
$$
\end{prop}
\begin{proof}
We reason with the map $\alpha$ (the argument with $\omega$ is similar). It suffices to show that the tangent maps $(T\alpha)_V:\TI_V\to ({\rm T}\p(\h))_{\alpha(V)}$,
$$
(T\alpha)_V(\v)=\v^*V+V^*\v
$$
are contractive. Pick a pair $(iX,iY)$, $X^*=X, Y^*=Y$ which lifts $\v$, i.e., $\v=iXV-iVY$.
Then
$$
(T\alpha)_V(\v)=(-iV^*X+iYV^*)V+V^*(iXV-iVY)=iY\alpha(V)-i\alpha(V)Y=i[Y,\alpha(V)].
$$
Note that the matrix of $[Y,\alpha(V)]$ in terms of the projection $\alpha(V)$ is
$$
\left(\begin{array}{cc} Y_{11} & Y_{12} \\ Y_{12}^* & Y_{22} \end{array}\right) 
\left(\begin{array}{cc} 1 & 0 \\ 0 & 0 \end{array}\right)-\left(\begin{array}{cc} 1 & 0 \\ 0 & 0 \end{array}\right)\left(\begin{array}{cc} Y_{11} & Y_{12} \\ Y_{12}^* & Y_{22} \end{array}\right)=\left(\begin{array}{cc} 0 & -Y_{12} \\ Y_{12}^* & 0 \end{array}\right),
$$
whose norm equals the norm of 
$$
\left(\begin{array}{cc} 0 & Y_{12} \\ Y_{12}^* & 0 \end{array}\right)=\left(\begin{array}{cc} 0 & -Y_{12} \\ Y_{12}^* & 0 \end{array}\right)\left(\begin{array}{cc} 1 & 0 \\ 0 & -1 \end{array}\right)=\left(\begin{array}{cc} 0 & -Y_{12} \\ Y_{12}^* & 0 \end{array}\right)(2\alpha(V)-1),
$$
bacause $2\alpha(V)-1$ is a unitary operator.
By Lemma \ref{area 51} 
$$
\Big\|\left(\begin{array}{cc} 0 & -Y_{12} \\ Y_{12}^* & 0 \end{array}\right)\Big\| \le \|Y\|\le \max\{\|X\|, \|Y\|\}=\|(X,Y)\|.
$$
Since this holds for any pair $(X,Y)$ which lifts $\v$, we have that
$$
|\v|_V=\inf\{\|(X,Y)\|: iXV-iVY=\v\}\ge \|(T\alpha)_V(\v)\|,
$$
as claimed. 
\end{proof}
\subsection{Balanced isometries}
Recall that $V\in\ii(\h)$ is called {\it balanced} if $n(V)=r^\perp(V)$. 
 
In some special directions, more can be said about minimal liftings and geodesics at balanced isometries. 
\begin{defi}
Let $V\in\ii(\h)$ be balanced, with inititial space $R(\alpha(V))=\s_i$ and final space $R(\omega(V))=\s_f$. We call a tangent vector $\v\in\TI_V$ {\rm orthogonal at } $V$ if $\v$  as an operator in $\h$, satisfies that $\v(\s_i)\perp\s_f$. Equivalently, $\omega(V)\v \alpha(V)=0$
\end{defi}
Let us construct explicit minimal liftings for tangent vectors which are othogonal to $V$.
\begin{enumerate}
\item
Let $\v\in\TI_V$ such that $\omega(V)V\alpha(V)=0$. First, using the left action of $\u(\h)\times\u(\h)$, we may suppose without loss of generality that $V=P_0$ is an orthogonal projection. Indeed, there exist $U,W\in\u(\h)$ such that $V=UP_0W^*$. Then $\v_0:=U^*\v W\in \TI_{P_0}$ satisfies
$$
0=\omega(V)\v\alpha(V)=VV^*\v V^*V=UP_0W^*(UP_0W^*)^*\v(UP_0W^*)^*UP_0W^*=UP_0U^*\v WP_0W^*
$$
$$
=UP_0\v_0P_0W^*,
$$
and thus, $P_0\v_0 P_0=0$. Suppose that we find $(A,B)$ a minimal lifting for $\v_0$ at $P_0$, i.e., $iAP_0-iP_0B=\v_0$ with 
$$
\|(A,B)\|\le \|(A,B)+(X,Y)\| \ \hbox{ for all } (iX,iY)\in\iota_{P_0}.
$$
Then $(UAU^*,W^*BW)$ is a lifting for $U\v_0W^*=\v$ at $V$: 
$$
iUAU^*V-iVW^*BW=U\{iAU^*VW^*-iU^*VWB\}W^*=U\{iAP_0-iP_0B\}W^*=U\v_0W^*=\v,
$$
which is minimal
$$
\|(UAU^*,W^*BW)\|=\|(A,B)\|\le\|(A,B)+(X,Y)\|=\|(UAU^*,W^*BW)+(UXU^*,W^*YW)\|,
$$
where $(UXU^*,W^*YW)$ parametrizes all elements in $Ad(U,W)(\iota_{P_0})=\iota_{UP_0W^*}=\iota_V$.
\item
Let us construct a minimal lifting for $\v_0$ at $P_0$ (with $P_0\v_0 P_0=0$). Let $(A,B)$, $A^*=A$, $B^*=B$ such that $iAP_0-iP_0B=\v_0$. Then $P_0AP_0=P_0BP_0$. Pick
$$
A_0=P_0AP_0^\perp+P_0^\perp AP_0  \ \hbox{ and } \ B_0=P_0BP_0^\perp+P_0^\perp BP_0.
$$
Clearly $A_0^*=A_0$ and $B_0^*=B_0$. Also, since $P_0AP_0=P_0BP_0$, after elementary computations,
$$
\v_0=iAP_0-iP_0B=iA_0P_0-iP_0B_0.
$$
Finally, if $(A',B')$ is another lifting of $\v_0$, then
$$
P_0\v_0P_0^\perp=P_0(iA'P_0-P_0B'P_0^\perp)P_0^\perp=iP_0B'P_0^\perp,
$$
i.e., $P_0B'P_0^\perp=P_0B_0P_0^\perp$, and therefore also 
$$
P_0^\perp B'P_0=(P_0B'P_0^\perp)^*=(P_0 B_0P_0^\perp)^*=P_0^\perp B_0P_0.
$$
That is, in matrices in terms of $P_0$, $B'$ and $B_0$ have the same off-diagonal entries. Clearly the same happens for $A'$ and $A_0$.
By Lemma \ref{area 51}, since $A_0$ and $B_0$ are codiagonal,
$$
\|A_0\|\le\| A'\| \ \hbox{ and } \ \|B_0\|\le \|B'\|,
$$
i.e., $(A_0,B_0)$ is a minimal lifting.
\end{enumerate}
These special (co-diagonal, minimal) liftings just exhibited for these special velocities, have  the following  property:
\begin{prop}
Let $V\in\ii(\h)$ and $\v\in\TI_V$ such that $\v$ is orthogonal to $\v$ (i.e., $\omega(V)\v\alpha(V)=0$). Pick $(A_0,B_0)$ a codiagonal minimal lifting of $\v$ as above. Then the curve $\delta(t)=e^{itA_0}Ve^{-itB_0}$ (minimal along its path up to $|t|\le\frac{\pi}{2|\v|_V}$), verifies that the  initial and final projection curves
$$
\alpha(\delta), \omega(\delta)\in\p(\h)
$$
are minimal along their paths in $\p(\h)$, for $|t|\le \frac{\pi}{2\|B_0\|}$ and $|t|\le \frac{\pi}{2\|A_0\|}$, respectively.
\end{prop}
\begin{proof}
Note that 
$$
\alpha(\delta)(t)=\delta^*(t)\delta(t)=(e^{itA_0}Ve^{-itB_0})^*e^{itA_0}Ve^{-itB_0}=e^{-itB_0}V^*Ve^{-itB_0}=e^{-itB_0}\alpha(V)e^{-itB_0},
$$
with $B_0$ co-diagonal with respect to $\alpha(V)$. Indeed, with the same argument as above, it suffices to consider  $V=P_0$, in which case it is evident. Therefore (see \cite{pr}), $\alpha(\delta)$ is minimal in $\p(\h)$ for $|t|\le\frac{\pi}{2\|B_0\|}$. The argument with $\omega(\delta)$ is analogous.
\end{proof}

In other words, for balanced isometries, and velocities which are orthogonal to $V$, locally, moving from $V_0$ to $V_1$ optimally in $\ii(\h)$,  involves the optimal paths for the initial and final spaces of $V_0$ and $V_1$.

Recall from Theorem \ref{comparacion} the comparison between the Finsler norm of $\v$ at $V$ and the ambient norm of $\v$: $|\v|_V\le\|\v\|$. Note  that, for  velocities which are orthogonal, at balanced partial isometries, both norms coincide:
\begin{rem}\label{coincidencia}
Let $V\in\ii(\h)$ and $\v\in\TI_V$ such that $\v$ is orthogonal to $V$. Then 
$$
|\v|_V=\|(A_0,B_0)\|=\max\{\|A_0\|,\|B_0\|\}=\|\v\|,
$$
its norm as an element in $\b(\h)$. Indeed, again it suffices to reason in the case $V=P_0$. 
As seen in the discussion  preceding the above proposition, $\v$ has co-diagonal matrix in terms of $P_0$:
$$
\v=\left(\begin{array}{cc} 0 & iB_0 \\ iA_0 & 0 \end{array}\right).
$$
Then 
$$
\|\v\|^2=\|\left(\begin{array}{cc} 0 & iB_0 \\ iA_0 & 0 \end{array}\right)^*\left(\begin{array}{cc} 0 & iB_0 \\ iA_0 & 0 \end{array}\right) \|=\|\left(\begin{array}{cc} A_0^2 & 0 \\ 0 & B_0^2 \end{array}\right)\|=\max\{\|A_0\|^2,\|B_0\|^2\}.
$$
\end{rem}
If $\gamma(t)\in\ii(\h)$, $t\in I$ is smooth, denote by $\ell_\infty(\gamma)$ the length of $\gamma$ with the metric induced by the ambient norm of $\b(\h)$:
\begin{equation}\label{metrica norma}
\ell_\infty(\gamma)=\int_I \|\dot{\gamma}(t)\| d t.
\end{equation}
\begin{coro}
Let $V\in\ii(\h)$ be balanced, and $\v\in\TI_V$ orthogonal to $V$. Let $(A,B)$ be a minimal lifting for $\v$. Then 
$$
\delta(t)=e^{itA}Ve^{-itB}
$$
is minimal along its path, for $|t|\le\frac{\pi}{2\|\v\|}$, when the lengths of curves are measured as in (\ref{metrica norma}), with the $\ell_\infty$ functional.
\end{coro}
\begin{proof}
By Corollary \ref{corolario 47}, we know that $\delta$ is minimal for the $\ell$ functional, for $|t|\le\frac{\pi}{2|\v|_V}=\frac{\pi}{2\|\v\|}$, because $|\v|_V=\|\v\|$. For an arbitrary smooth curve $\gamma$ in $\ii(\h)$, Theorem \ref{comparacion} implies that
$$
\ell_\infty(\gamma)=\int_I\|\dot{\gamma}(t)\| d t\le \int_I|\dot{\gamma}(t)|_{\gamma(t)} d t =\ell(\gamma).
$$
Note that
$$
\dot{\delta}(t)=e^{itA}\{iAV-iVB\}e^{-itB}=e^{itA}\v e^{-itB}.
$$
Then, the facts that the metric $|\ \ |_V$ is invariant under the action of $\u(\h)\times\u(\h)$, and that $\v$ is orthogonal, imply that 
$$
|\dot{\delta}(t)|_{\delta(t)}=|\v|_V=\|\v\|=\|e^{itA}\v e^{-itB}\|=\|\dot{\delta}(t)\|,
$$
and therefore $\ell(\delta)=\ell_\infty(\delta)$.
\end{proof}
Clearly $\p(\h)$ is a complemented submanifold of $\ii(\h)$. We shall prove another consequence of Remark \ref{coincidencia}:  if $P_0, P_1\in\p(\h)$ are regarded as points in $\ii(\h)$, and they can be joined by a minimal  geodesic in $\p(\h)$, then this path is minimal between $P_0$ and $P_1$ in $\ii(\h)$.

Before, let us  recall the necessary and sufficient condition that $P_0, P_1$ must satisfy in order that they can be joined by a minimal geodesic of $\p(\h)$ (see \cite{p-q}):
\begin{rem}
Let $P_0,P_1\in\p(\h)$, then there exists a minimal geodesic of $\p(\h)$ (or in fact, any geodesic) joining $P_0$ and $P_1$ if and only if 
\begin{equation}\label{condicion p}
\dim\left(R(P_0)\cap N(P_1)\right)=\dim\left(R(P_1)\cap N(P_0)\right).
\end{equation}
\end{rem}

\begin{coro}
Let $P_0, P_1\in\p(\h)$ satisfy condition {\rm (\ref{condicion p})}. Let $\delta(t)\in\p(\h)$ be a geodesic joining $\delta(0)=P_0$ and $\delta(1)=P_1$, and $\gamma(t)\in\ii(\h)$ be any other smooth  curve joining $P_0$ and $P_1$. Then
$$
\ell(\delta)\le\ell(\gamma).
$$
\end{coro}
\begin{proof}
First, note that if $P$ is a projection and $\v\in\left({\rm T}\p(\h)\right)_P$, then $\v$ is $P$-co-diagonal: $P\v P=0$ (or, in the notation employed here, $\v$ is {\it orthogonal at} $P$, regarded as an element in $\ii(\h)$). This basic fact is well known in the geometry of $\p(\h)$ (see \cite{cpr}): if $P(t)$ is a smooth curve in $\p(\h)$ with $P(0)=P$ and $\dot{P}(0)=\v$, then differentiating $P^2(t)=P(t)$ yields (at $t=0$)
$$
\v P+P\v=\v,
$$
which implies $P\v P=P^\perp \v P^\perp=0$. Therefore, by Remark \ref{coincidencia}, $|\v|_P=\|\v\|$. It follows that $\ell_\infty(\delta)=\ell(\delta)$. On the other hand, by Proposition \ref{contracciones},
$$
\ell_\infty\left(\alpha(\gamma)\right)\le \ell(\gamma);
$$
since $\delta$ is minimal in $\p(\h)$,
$$
\ell(\delta)=\ell_\infty(\delta)\le\ell_\infty\left(\alpha(\gamma)\right),
$$
and the proof follows.
\end{proof}

{\bf Acknowledgements}

This work was supported by the grant PICT 2019 04060 (FONCyT - ANPCyT, Argentina)

\end{document}